\documentclass[reqno]{amsart}
\usepackage{amssymb}
\usepackage{amsmath}
\usepackage{enumitem}
\usepackage{mathrsfs}
\usepackage[all]{xy}
\usepackage{hyperref}
\usepackage[margin=1.3in]{geometry}

\numberwithin{equation}{section} \numberwithin{figure}{section}

\DeclareMathOperator{\Pic}{Pic} 
\DeclareMathOperator{\Gal}{Gal} \DeclareMathOperator{\NS}{NS}
 \DeclareMathOperator{\Spec}{Spec}
 \DeclareMathOperator{\rank}{rank}
\DeclareMathOperator{\Hom}{Hom} \DeclareMathOperator{\re}{Re}
   
\DeclareMathOperator{\vol}{vol} 
 \DeclareMathOperator{\Val}{Val}
\DeclareMathOperator{\Eff}{Eff} \DeclareMathOperator{\Res}{R_{F/E}}

\newcommand{\Fbar}{\overline{F}}
\newcommand{\Ebar}{\overline{E}}
\newcommand{\Xbar}{\overline{X}}
\newcommand{\Lbar}{\overline{L}}
\newcommand{\LL}{\mathcal{L}}

\newcommand{\Norm}[2]{\operatorname{N}_{{{\mathrm{#1}}/{\mathrm{#2}}}}}
\newcommand{\RRes}[2]{\operatorname{R}_{{{\mathrm{#1}}/{\mathrm{#2}}}}}
\newcommand{\Adele}{\mathbf{A}}
\newcommand{\pp}{\mathfrak{p}}
\newcommand{\aaa}{\mathfrak{a}}
\newcommand{\cc}{\mathfrak{c}}

\newcommand\PP{\mathbb{P}}
\newcommand\ZZ{\mathbb{Z}}
\newcommand\NN{\mathbb{N}}
\newcommand\QQ{\mathbb{Q}}
\newcommand\RR{\mathbb{R}}
\newcommand\CC{\mathbb{C}}
\newcommand\GG{\mathbb{G}}

\newcommand\Gm{\GG_\mathrm{m}}
\newcommand{\Mob}{M\"{o}bius }

\newtheorem{lemma}{Lemma}

\newtheorem{theorem}[lemma]{Theorem}

\theoremstyle{definition}
\newtheorem{example}[lemma]{Example}
\newtheorem{definition}[lemma]{Definition}

\numberwithin{lemma}{section}

\title[Rational points and Weil restriction]{Rational points of bounded height and the Weil restriction}
\author[Daniel Loughran]{Daniel Loughran}
\address{
Leibniz Universit\"{a}t Hannover \\
Institut f\"{u}r Algebra, Zahlentheorie
    und Diskrete Mathematik, \\
Welfengarten 1 \\
30167 Hannover \\
Germany.}
\email{loughran@math.uni-hannover.de}

\begin{document}

\begin{abstract}
    Given an extension of number fields $E \subset F$ and a projective variety $X$ over $F$,
    we compare the problem of counting the number of rational points of bounded height on $X$ with that
    of its Weil restriction over $E$. In particular, we consider the compatibility with respect to the Weil restriction of conjectural asymptotic formulae
	due to Manin and others. Using our methods we prove several new cases of these conjectures. We also construct new counterexamples
	 over every number field.
\end{abstract}

\maketitle

\tableofcontents

\section{Introduction}
Let $X$ be a smooth projective variety over a number field $F$. To any embedding $X \subset \PP_F^n$ of $X$ over $F$, we may associate a height function given by
\begin{equation} \label{eq:height}
    H(x)= \prod_{v \in \Val(F)}\max\{|x_0|_v,\ldots,|x_n|_v\},
\end{equation}
where $x=(x_0:\cdots:x_n) \in X(F)$ and  $|\cdot|_v$ is the usual absolute value associated to a place $v$ of $F$. The product formula $\prod_{v \in \Val(F)} |\lambda|_v=1$, for any $\lambda \in F^*$, implies that this  expression is independent of the choice of representation of $x$ in homogeneous coordinates. More generally, one may associate a height function $H_{\LL}$ to any adelically metrised line bundle $\LL=(L,||\cdot||)$ on $X$ (see Section \ref{Sec:adelic} for further details). The advantage of such a definition is that it is intrinsic, i.e it does not depend on a choice of embedding. In the case where $L$ is ample, the number of rational points of bounded height is finite and thus it makes sense to consider the counting function
$$N(\LL,U,B)=\#\{x \in U(F) : H_{\LL}(x) \leq B\},$$
for any $B>0$ and any open subset $U \subset X$. More generally still, if $L$ is big then the number of rational points of bounded height is finite on some open subset of $X$, thus we also obtain well-defined counting functions on certain open subsets of $X$. One can even define these counting functions for arbitrary adelically metrised line bundles $\LL$, where by convention if $U$ contains infinitely any rational points of bounded height we set $N(\LL,U,B) = \aleph_0$.

In the papers \cite{FMT89} and \cite{BM90}, Manin and his collaborators formulated various conjectures on the asymptotic behavior of these counting functions as $B \to \infty$. If we denote by $\Eff(X) \subset \NS_{\RR} X = \NS X \otimes_{\ZZ} \mathbb{R}$
the closed cone generated by the classes of effective divisors, then we define the Nevanlinna invariant of an effective line bundle $L$ on $X$ to be
\begin{equation}
	a(L)=\inf \{ r \in \QQ : r[L] + [\omega_X] \in \Eff(X)\}. \label{def:a(L)}
\end{equation}
Here $[L]$ denotes the class of $L$ in $\NS_{\RR} X$ and $\omega_X$ denotes the canonical bundle of $X$. Then, Manin and his collaborators conjectured that if $L$ is ample then for any $\varepsilon > 0$ there exists an open subset $U \subset X$ such that
\begin{equation}
	N(\LL,U,B) \ll_{\varepsilon,U, \LL} B^{a(L) + \varepsilon}, \label{conj:1}
\end{equation}
as $B \to \infty$. Note that in general one needs to restrict to some open subset in order to avoid ``accumulating subvarieties". For example,
a line on a smooth cubic surface $S \subset \PP^3_F$ contains roughly $B^2$ rational points of height less than $B$, whereas one has $a(\mathcal{O}_S(1))=1$.
They also conjectured a more precise asymptotic formula in the case where $[\omega_X]$ is not effective. Namely that, possibly after a finite field extension, there exists an open subset $U \subset X$ and a positive constant $c=c(\LL,U)$ such that
\begin{equation}
	N(\LL,U,B)= c B^{a(L)} (\log B)^{b(L)-1}(1 + o(1)), \label{conj:2}
\end{equation}
as $B \to \infty$, where $b(L)$ is the codimension of the minimal face of $\partial\Eff(X)$ which contains $a(L)[D] + [\omega_X]$. Note that one needs to assume that $a(L)[D] + [\omega_X]$ belongs to the polyhedral part of $\partial\Eff(X)$ for $b(L)$ to be well-defined.
The leading constant in this conjecture has also received a conjectural interpretation due to Peyre \cite{Pey95} in the case where $\omega^{-1}_X$ is ample,
in which case one has $a(\omega^{-1}_X)=1$ and $b(\omega^{-1}_X)=\rank \Pic X$.

However it turns out that this second conjecture (\ref{conj:2}) as stated is not true, and an explicit family of counterexamples over certain number fields was constructed by Batyrev and Tschinkel \cite{BT96}. Nevertheless (\ref{conj:2}) is still true in many cases, and sometimes in more generality
than originally stated (for example with $L$ big, rather than just ample). For example (\ref{conj:2}) is known for some del Pezzo surfaces (e.g. \cite{BB11} and \cite{Lou10}), flag varieties \cite{FMT89}, toric varieties \cite{BT98}, and various other equivariant compactifications of homogeneous spaces \cite{CLT02}. These conjectures have also been shown to be compatible with various geometrical constructions, such as products. However one important construction seems to have been so far overlooked, namely the \emph{Weil restriction}.

The Weil restriction (or restriction of scalars) was originally defined by Weil \cite{Wei82}, and is a way of constructing an algebraic variety
$\Res X$ over a smaller number field $E \subset F$ whose set of $E$-rational points is canonically in bijection with the $F$-rational
points of $X$ (see Section \ref{Sec:Weil} for a precise definition). In general, one expects the arithmetic properties of $X$ to be closely related to those of $\Res X$. For example if $A$ is an
abelian variety, then so is $\Res A$ and Milne \cite{Mil72} has shown that the Birch and Swinnerton-Dyer conjecture holds for $A$ if and only if it
holds for $\Res A$.
In this paper we address the question of how the counting problems for $X$ and $\Res X$ compare.

The first problem is to construct a height function on $\Res X$ from one on $X$. In Section \ref{Subsec:Weil_adelic_metric} we
show that given an adelically metrised line bundle $\LL$ on $X$,
there is a way to define an adelically metrised line bundle $\Res \LL$ on $\Res X$ which satisfies
\begin{equation}
    N(\LL,U,B)=N(\Res \LL,\Res U,B), \label{eqn:Weil_counting_functions}
\end{equation}
for any open subset $U \subset X$ and any $B>0$. Moreover this construction preserves positivity properties (such as effectiveness, ampleness and bigness) and also the canonical line bundle. This allows us to define the Weil restriction $\Res H_{\LL}$ of the associated height function $H_{\LL}$.
This leads to our main theorem.

\begin{theorem} \label{thm:Main}
    Let $E \subset F$ be number fields. Let $(X,\LL)$ be a smooth projective variety over $F$
    together with a big adelically metrised line bundle such that $X(F) \neq \emptyset$. 
    Let $\varepsilon > 0$ and let $U \subset X$ be an open subset.
    Then (\ref{conj:1}) holds for $(X,U,\LL)$ if and only if (\ref{conj:1}) holds for $(\Res X, \Res  U,\Res \LL)$.

    Moreover, if $[\omega_X]$ is not effective and $H^1(X,\mathcal{O}_X)=0$, then (\ref{conj:2}) holds for $(X,U,\LL)$ if and only if (\ref{conj:2}) holds for $(\Res X,\Res U,\Res \LL)$.
\end{theorem}

Examples of varieties for which $[\omega_X]$ is not effective and $H^1(X,\mathcal{O}_X)=0$ include all rationally connected varieties \cite[Cor.~4.18]{Deb01}, in particular all geometrically rational varieties and all Fano varieties. For these latter classes of varieties we also show that the refined conjecture due to Peyre \cite{Pey95}, on the leading constant appearing in the asymptotic formula, is compatible with the Weil restriction.

An immediate corollary of Theorem \ref{thm:Main} is that if Manin's conjectures hold for $(X,\LL)$, then they also hold for $(\Res X, \Res \LL)$. The problem with the converse is that the open subset $U' \subset \Res X$ for which (\ref{conj:1}) or (\ref{conj:2}) holds might not be of the form $\Res U$ for
some open subset $U \subset X$. Nevertheless, there are many examples where this is true.
For example, for flag varieties (\ref{conj:2}) holds on the whole space, i.e. it is not necessary to
restrict to an open subset. Therefore the equivalence of Manin's conjecture for a flag variety $X$ and for its Weil restriction $\Res X$, which is also a flag variety, is clear.
Hence we see that Manin's conjecture for all choices of adelic metric on every big line bundle on every flag variety over $\QQ$ is equivalent to the same conjecture for all flag varieties \emph{over any number field}.
For toric varieties, Manin's conjecture is known to hold on the open subset given by the embedded algebraic torus $T$. As the Weil restriction $\Res X$ of such a toric variety $X$ is a toric variety under the algebraic torus $\Res T$, we again see that one may reduce the proof of Manin's conjecture for all
toric varieties over every number field to those toric varieties which are defined over $\QQ$. Similar results hold for equivariant compactifications of other homogeneous spaces and for varieties for which every accumulating subvariety of $\Res X$ is of the form $\Res Z$ for some subvariety $Z \subset X$.
Using Theorem \ref{thm:Main} we are also able to obtain many new cases of Manin's conjecture,
given as the Weil restrictions of suitable complete intersections.

\begin{theorem}\label{thm:CI_Weil}
	Let $E\subset F$ be number fields and let $X \subset \PP^n$ be a non-singular complete intersection of $m$
	hypersurfaces over $F$ each of the same degree $r$. Suppose that $$n \geq  m(m+1)(r-1)2^{r-1} + m,$$
	and $X(\Adele_F) \neq \emptyset$. Let $H$ be the height function (\ref{eq:height}) on $X$. 
	Then Manin's conjecture (\ref{conj:2}) with Peyre's constant holds for $\Res X$ with respect to $\Res H$.
\end{theorem}
We are in fact able to handle more general height functions than (\ref{eq:height}), namely we allow
arbitrary norms at the archimedean places, rather than simply the maximum norm (see Section \ref{Sec:CI}
for a precise statement).
The varieties $\Res X$ occurring in Theorem \ref{thm:CI_Weil} are complete intersections in the Weil
restriction $\Res \PP^n$ of projective space. This result therefore follows the philosophy
emphasised in \cite{Pey01}, namely that of studying Manin's conjecture for complete intersections
inside arbitrary Fano varieties, rather than just complete intersections in the usual projective space.
In \cite{Ski97}, Skinner used the circle method to prove that weak approximation holds 
for the complete intersections $X$ occurring in Theorem \ref{thm:CI_Weil},
by counting rational points in certain ``boxes''. Skinner's boxes are quite different 
however from the regions cut out by height functions, with the outcome being that Skinner's main theorem does not directly imply Manin's conjecture.
In order to prove Theorem \ref{thm:CI_Weil}, we show that Skinner's result
may indeed be used to prove Manin's conjecture for such complete intersections.
Our proof proceeds by covering the region of interest with Skinner's boxes and then applying his result to each such box.
Other new cases of Manin's conjecture may be obtained by applying Theorem \ref{thm:Main} to the 
del Pezzo surfaces over number fields recently considered in
\cite{DF13a}, \cite{DF13b}, \cite{DF14} and \cite{FP13}.

As we have already noted, Manin's conjecture (\ref{conj:2}) is not true in general and a family of counterexamples
was constructed in \cite{BT96}. However these counterexamples were only constructed over those number fields
which contain $\QQ(\sqrt{-3})$, in particular the existence of counterexamples over $\QQ$ was left open.
In Section \ref{Sec:counterexample} we apply the Weil restriction to the construction of \cite{BT96} to produce counterexamples
to Manin's conjecture (\ref{conj:2}) over \emph{any} number field.

\begin{theorem} \label{thm:counterexample}
    Let $E$ be a number field. Then there exists a Fano variety $X$ over $E$ such that
    for every number field $E \subset F$, every open subset $U \subset X_F$ and every
    choice of adelic metric on $\omega_{X_F}^{-1}$ we have
    $$N(\omega_{X_F}^{-1},U,B) \gg B(\log B)^{\rho(X_F)+1},$$
    as $B \to \infty$, where $\rho(X_F)= \rank \Pic X_F$.
\end{theorem}

After Milne's article \cite{Mil72}, the Weil restriction became a useful tool for researchers
working on the Birch and Swinnerton-Dyer conjecture. The author hopes that the same will become
true of Manin's conjectures.

The layout of this paper is as follows. In Section \ref{Sec:Weil} we recall certain facts about the Weil restriction and also define the Weil restriction of
a line bundle. Section \ref{Sec:adelic} contains various results on adelically metrised line bundles and height functions, and we also define
the Weil restriction of an adelically metrised line bundle. We finish the paper by proving our main theorems in Section \ref{Sec:Manin}, together with the fact that
Peyre's conjectural constant is compatible with the Weil restriction.

\medskip
\noindent\textbf{Acknowledgments:}
The majority of this work was completed whilst the author was working at l'Institut de Math\'{e}matiques de Jussieu
and supported by ANR PEPR.
The author would like to thank Tim Browning, Marc Hindry, Emmanuel Peyre and Tomer Schlank for useful comments and ideas.

\subsection*{Notation}
\subsubsection*{Geometry}
For a field $F$, we denote by $\PP_F^n$ and $\mathbb{A}_F^n$ projective $n$-space and affine $n$-space over $F$ respectively. We sometimes omit the subscript $F$ if the field is clear.
A variety over $F$ is a separated geometrically integral scheme of finite type over $F$. For every field, we fix a choice of algebraic closure $\Fbar$ and we denote by $G_F$ the absolute Galois group of $F$
with respect to $\Fbar$.

By a line bundle, we mean a locally free sheaf of rank one. Given a line bundle $L$ on a scheme $X$ over a field $F$,
we denote by $\Xbar$ and $\Lbar$ the base change of $X$ and $L$ to $\Fbar$, respectively. We denote by $\Pic^0 X$ the subgroup of
$\Pic X$ of line bundles which are algebraically equivalent to $\mathcal{O}_X$, and by $\Pic^L X$ the subset of $\Pic X$ consisting of those line bundles which are algebraically equivalent to a fixed line bundle $L$.
Note that $\Pic^L X$ is a torsor for $\Pic^0 X$. Given a line bundle $L$ on a variety $X$, we denote by $[L]$ the class of $L$ in $\NS_{\RR} X$ and by $\omega_X$ the canonical bundle of $X$
if $X$ is also smooth.
The symbol $\boxtimes$ is used to denote the external tensor product. Namely, given line bundles $L_i$ on varieties $X_i$ $(i=1,2)$,
we define $L_1 \boxtimes L_2 = \pi^*_1 L_1 \otimes \pi^*_2 L_2$ as a line bundle on $X_1 \times X_2$, where $\pi_i:X_1 \times X_2 \to X_i$ denote the canonical projections $(i=1,2)$.

\subsubsection*{Number theory}
For any number field $F$, we denote by $\mathcal{O}_{F}$ the ring of integers of $F$ and by $\Val(F)$ the set of valuations of $F$. For any $v \in \Val(F)$, we denote by $F_v$ (resp. $\mathcal{O}_{F_v}$)  the completion of $F$ (resp. $\mathcal{O}_{F}$) with respect to $v$. Given a finite set of places $S \subset \Val(F)$ containing all archimedean places, we denote by
$\mathcal{O}_{F,S}$ the ring of $S$-integers of $\mathcal{O}_{F}$.
We choose absolute values on each $F_v$ such that $|x|_v=|N_{F_v/\QQ_p}(x)|_p$,
where $v|p \in \Val(\QQ)$ and $|\cdot|_p$ is the usual absolute value on $\QQ_p$.
The advantage of these choices is that we have the
following product formula
$$\prod_{v \in \Val(F)}|x|_v =1 , \quad \text{for all } x \in F^\times.$$
We denote by $F_\infty = F \otimes_\QQ \RR=\prod_{v \mid \infty} F_v$.
We also choose algebraic closures $F_v \subset \Fbar_v$ and we equip $\Fbar_v$ with the unique
absolute value extending the absolute value on $F_v$.
We choose Haar measures $\mathrm{d}x_v$ on each $F_v$
such that $$\int_{\mathcal{O}_{F_v}}\mathrm{d}x_v = 1,$$
for all but finitely many archimedean $v$. We equip the adeles
$\Adele_F$ of $F$ with the induced Haar measure and denote by 
$\mu_F$ the volume of $\Adele_F/F$ with respect to the induced quotient measure.

\section{The Weil restriction} \label{Sec:Weil}
We begin by recalling the definition of the Weil restriction. The Weil restriction was originally defined by Weil in \cite{Wei82} (which he called the restriction of scalars), however we follow a more modern approach as can be found in \cite[Ch.7.6]{BLR90}.

Let $A$ be a commutative ring and let $B$ be an $A$-algebra which as an $A$-module is finite and locally free (e.g. $A \subset B$ is a finite field extension).
For any scheme $X$ over $B$ we define the functor $\RRes{B}{A} X$, from the dual of the category of schemes defined over $A$ to the category
of sets, to be the right adjoint of base change. That is to say, we define
$$\RRes{B}{A} X (S) = X(S \times_A B),$$
for any $A$-scheme $S$. If this functor is representable by a scheme over $A$, then we call this scheme (also denoted by $\RRes{B}{A} X$) the Weil restriction of $X$.
For our purposes, it will be sufficient to know that the Weil restriction exists whenever $X$ is quasi-projective over $B$ (see \cite[Thm.~7.6.2]{BLR90}).
Moreover if $B$ is \'{e}tale (e.g. $A \subset B$ is a finite separable field extension), then the Weil restriction of an affine, projective or smooth scheme is also affine, projective or smooth, respectively
\cite[Prop.~7.6.5]{BLR90}. The assignment of the Weil restriction can be viewed as a functor $\RRes{B}{A}$ in its own right and this functor preserves open and closed immersions and fibre products \cite[Prop.~7.6.2]{BLR90}.

We shall be particularly interested in the case of a finite extension $E \subset F$ of perfect fields of degree $d$.
We denote by $\sigma_i:F \to \Ebar$ the embeddings of $F$ into $\Ebar$ for $i=1,\ldots,d$. Then given a scheme $X$ over $F$, the counit of the adjunction gives rise to a morphism $p: \Res X \to X$ defined over $F$ which induces an isomorphism $$P:=\prod_{\sigma}p^{\sigma}: \Res X \to \prod_{i=1}^d X^{\sigma_i},$$
over $\Ebar$, where $X^{\sigma_i}=X \times_{\sigma_i} \Ebar $ denotes the conjugate of $X$ with respect to $\sigma_i$.
\begin{example}
    \begin{enumerate}
        \item The Weil restriction of the affine line $\mathbb{A}^1_F$ over $F$ is the affine space $\mathbb{A}^d_E$ over $E$.
        The morphism $p$ can be realised as
        $$
            \mathbb{A}^d_F \to \mathbb{A}^1_F, \qquad
            (x_1,\ldots,x_d) \mapsto \sum_{i=1}^d \alpha_i x_i,
        $$
        where $\alpha_1,\ldots,\alpha_d$ is a choice of basis for the field extension $E \subset F$.
        Given that the functor $\Res$ preserves affine varieties, fibre products
        and closed embeddings, this gives a simple way to write down equations for the Weil restrictions of affine varieties.
        \item Equations for the Weil restrictions of projective varieties are not as simple in general.
        For example,
        if $E\subset F$ is a quadratic field extension and $X=\PP^1_F$,
        then $\Res X$ can be embedded as a quadratic surface in $\PP^3_E$. Indeed, $\Res X$
        is isomorphic to $\PP^1_F \times \PP_F^{1\sigma} \cong \PP^1_F \times \PP^1_F$ over $F$, where $\sigma$
        is the non-trivial element of $\Gal(F/E)$. If $(x,y) \in \Res X(E)$, then the two divisors
        $L_1=\PP^1_F \times \{y\}$ and $L_2=\{x\} \times \PP^{1\sigma}_F$ are swapped by $\Gal(F/E)$. Hence the divisor
        $L_1 + L_2$ is defined over $E$ and moreover gives the required embedding $\Res X \hookrightarrow \PP_E^3$.
        For general $d$, a similar argument shows that equations for $\Res \PP^n_F$ can be given by some appropriate twist of
        the Segre embedding of $\prod_{i=1}^d \PP^n_E$.
    \end{enumerate}
\end{example}

\subsection{The norm of a line bundle}
We now recall some facts that we shall need on the norm of a line bundle (see e.g. \cite[Sec. 6.5]{EGAII} or \cite[Sec. 4.1]{Oes84}). Let $A$ be a commutative ring and let $B$ be an $A$-algebra
which as an $A$-module is finite and  locally free of rank $d$ (e.g. $A \subset B$ is a finite field extension of degree $d$).
Let $X_A$ be a reduced Noetherian scheme of finite type over $A$ and let $L$ be a line bundle on $X_B=X_A \times_A B$. Then
if $f:X_B \to X_A$ denotes the base change map, it follows that $f_*L$ is a vector bundle of rank $d$ on $X_A$. We define
$$\Norm{B}{A}(L) = \Hom_{\mathcal{O}_{X_A}}(\det f_*\mathcal{O}_{X_B}, \det f_*L),$$
which is a line bundle on $X_A$. We have the following properties.
\begin{enumerate}
    \item There is a canonical isomorphism $\Norm{B}{A} (\mathcal{O}_{X_B}) \cong \mathcal{O}_{X_A}$.
	\item The norm functor respects base change, i.e. if $A'$ is an $A$-algebra then we have a canonical isomorphism
		$$ \Norm{B}{A}( L) \otimes_A A' \cong \Norm{B'}{A'}( L), $$
		where $B' = A' \otimes_A B$. In particular if $E \subset F$ is a finite field extension of perfect fields, there is a canonical isomorphism
		$\overline{\Norm{F}{E}(L)} \cong \otimes_{i=1}^d L^{\sigma_i}$ over $\Ebar$.
	\item If $L_1$ and $L_2$ are two line bundles on $X_B$, then
    $$\Norm{B}{A} (L_1 \otimes L_2) \cong \Norm{B}{A}( L_1) \otimes \Norm{B}{A}( L_2),$$ canonically.
	In particular we obtain an induced homomorphism $\Norm{B}{A}:\Pic X_B \to \Pic X_A$ of Picard groups.
	\item To any local section $s$ of $L$ is associated a local section $\Norm{B}{A}(s)$ of $\Norm{B}{A}(L)$.
\end{enumerate}
Note that our notation differs slightly from that of \cite[Sec. 6.5]{EGAII}, where the norm is defined for more
general finite morphisms of schemes. We have used different notation here, as we shall only be taking the norm
with respect to finite morphisms arising from base change.

\subsection{The Weil restriction of a line bundle}
We now define the Weil restriction of a line bundle. Throughout this section $E \subset F$ is a finite extension of perfect fields of degree $d$. We denote by
$\sigma_i:F \to \Ebar$ the embeddings of $F$ into $\Ebar$ $(i=1,\ldots,d)$.
We define the Weil restriction of a line bundle $L$ on a quasi-projective variety $X$ over $F$ to be
\begin{equation} \label{def:Weil_Res_Line_bundle}
    \Res L = \Norm{F}{E} (p^* L).
\end{equation}
Note that over $\Ebar$ we have isomorphisms 
$$\overline{\Res L} \cong \otimes_{i=1}^d (p^* L)^{\sigma_i} \cong P^*(\boxtimes_{i=1}^d L^{\sigma_i}).$$
This induces an injective homomorphism $\Res : \Pic X \to \Pic \Res X$. Indeed, 
the fact that it is a homomorphism follows from the properties of the norm map
and the pull-back. To see that it is injective, note that if $\Res L \cong \mathcal{O}_{\Res X}$,
then $\boxtimes_{i=1}^d L^{\sigma_i} \cong \mathcal{O}_{\prod_{i=1}^d X^{\sigma_i}}$
and hence $L^{\sigma_i} \cong \mathcal{O}_{X^{\sigma_i}}$. Thus $L \cong \mathcal{O}_X$, as required.
We similarily obtain injective homomorphisms $\Res: \Pic^0 X \to \Pic^0 \Res X$ and $\Res: \NS X \to \NS \Res X.$
For any local section $s$ of $L$, we also obtain a local section $\Res s=\Norm{F}{E} (p^*s)$ of $\Res L$.

\begin{lemma} \label{lem:line_bundles}
Let $(X,L)$ be a smooth projective variety over $F$ together with a line bundle. Then
\begin{enumerate}
    \item $\Res \omega_X \cong \omega_{\Res X}$, where $\omega_X$ denotes the canonical line bundle of $X$. \label{item:canonical}
    \item $L$ is effective (resp. big, resp. ample) if and only if the same holds for $\Res L$. \label{item:positivity}
\end{enumerate}
\end{lemma}
\begin{proof}
In what follows, we identify $\overline{\Res X}$ with $\prod_{i=1}^d X^{\sigma_i}$ and $\overline{\Res L}$ with $\boxtimes_{i=1}^d L^{\sigma_i}$.
To prove the first part of the lemma, we note that given non-singular varieties $X_j$ for $j=1,2$, we have $\omega_{X_1\times X_2}\cong\omega_{X_1} \boxtimes \omega_{X_2}$ \cite[Ex. II.8.3]{Har77}. Therefore, we see that the canonical line bundle of $\prod_{i=1}^d X^{\sigma_i}$ is isomorphic to $\boxtimes_{i=1}^d \omega_{X^{\sigma_i}}$ and (\ref{item:canonical}) follows.

By the K\"{u}nneth formula for coherent cohomology \cite{SW59} and flat base change \cite[Prop.~III.9.3]{Har77}, we have
$$h^0(\Res X,\Res L) = h^0(X,L)^d.$$
From this, we see that $h^0(\Res X,\Res L)\neq 0$ if and only if $h^0(X,L) \neq0$, i.e. $L$ is effective if and only if $\Res L$ is effective. Similarly, as the property of being big can be defined in terms of the size of the space of global sections \cite[Lem. 2.2.3]{Lar07}, it follows that $L$ is big if and only if $\Res L$ is big.

Next, let $\varphi:X \dashrightarrow \PP_F^n$ be a rational map associated to $L$. Then, a rational map associated to $\Res L$ may be given by the composition of $\prod_{i=1}^d \varphi^{\sigma_i}$ with the Segre embedding, on choosing isomorphisms $(\PP_{\Ebar}^n)^{\sigma_i} \cong \PP_{\Ebar}^n$. As this map is an embedding if and only if $\varphi$ is an embedding, we see that $L$ is very ample if and only if $\Res L$ is very ample, and therefore that $L$ is ample if and only if $\Res L$ is ample. This proves (\ref{item:positivity}).
\end{proof}

We now study the relationship between $\Pic X$ and $\Pic \Res X$. For this we shall often use the following well-known result.

\begin{lemma}\label{lem:CMPic}
	Let $X$ be a proper variety over $F$. If $X(F) \neq \emptyset$ then
	the natural map
	$$\Pic X \to (\Pic \Xbar)^{G_F},$$
	is an isomorphism.
\end{lemma}
\begin{proof}
	See \cite{CM96}, in particular \cite[Cor.~1.3]{CM96}.
\end{proof}

Throughout this paper we will often assume that our varieties have rational points in order to apply this lemma. The existence of a rational point will also be
crucial when we show that Tamagawa measures are preserved under the Weil restriction (see Lemma \ref{lem:cotangent}).

\begin{lemma} \label{lem:Pic0}
	Let $X$ be a smooth projective variety over $F$ such that $X(F) \neq \emptyset$.
	Then the map $\Res: \Pic^0 X \to \Pic^0 \Res X$ is an isomorphism.
\end{lemma}
\begin{proof}
	We again identify $\overline{\Res X}$ with $\prod_{i=1}^d X^{\sigma_i}$ and $\overline{\Res L}$ with $\boxtimes_{i=1}^d L^{\sigma_i}$ for any line bundle $L$ on $X$.
  	First, it is well-known that for smooth projective varieties $X_1$ and $X_2$ over an algebraically closed field the natural map $\Pic^0 X_1 \bigoplus \Pic^0 X_2 \to \Pic^0(X_1 \times X_2)$
	is an isomorphism (see e.g. \cite[Prop.~A.4]{Diem01}). In particular the map
    \begin{align*}
        &\bigoplus_{i=1}^d \Pic^0 X^{\sigma_i} \to \Pic^0 \left(\prod_{i=1}^d X^{\sigma_i}\right)\\
        &(L_1^{\sigma_1},\ldots,L_d^{\sigma_d}) \mapsto \boxtimes_{i=1}^d L_i^{\sigma_i},
    \end{align*}
	is an isomorphism. 	This map is obviously a homomorphism of $G_F$-modules and thus shows that $\Pic^0 \:\overline{\Res X}$ 
	is the representation induced from the action of $G_F$ on
	$\Pic^0 \Xbar$. Next, by Shapiro's Lemma \cite[Prop.~1.6.3]{NSW00} we see that $(\Pic^0 \: \overline{\Res X})^{G_E} = (\Pic^0 \Xbar)^{G_F}$. 
	As $X(F) \neq \emptyset$ and $\Res X(E) \neq  \emptyset$, we have the equalities $\Pic^0 X = (\Pic^0 \Xbar)^{G_F}$
	and $\Pic^0 \Res X=(\Pic^0 \: \overline{\Res X})^{G_E}$ by Lemma \ref{lem:CMPic}, and the result follows.
\end{proof}

\begin{example}
    We sketch an example which shows that the map $\Res: \Pic X \to \Pic \Res X$ 
    need not be an isomorphism in general.
	If $E \subset F$ has degree two and $C$ is an elliptic curve over $E$, then $\Res C_F$ is isogenous
    to $C \times C'$ over $E$, where $C'$ denotes the quadratic twist of $C$ with respect to
    $E \subset F$. In particular, the pull-back of $C \times {0}$ and ${0} \times C'$ give two linearly independent curves in $\NS(\Res C_F)$.
     Thus clearly $\Pic C \not \cong \Pic \Res C$; indeed $C$ has Picard number one whereas $\Res C_F$ has Picard number at least two.
\end{example}

However in the case where $X$ is Fano, or more generally when $H^1(X, \mathcal{O}_X)=0$, the map $\Res: \Pic X \to \Pic \Res X$ \emph{is} an isomorphism as soon as there is a rational point.
\begin{lemma}\label{lem:Picard}
	Let $X$ be a projective variety over $F$
	such that $H^1(X, \mathcal{O}_X)=0$. Then
	\begin{enumerate}
    		\item $H^1(\Res X, \mathcal{O}_{\Res X})=0$. \label{h^1_structure_sheaf}
	    \item There is an isomorphism $\Pic \Xbar \otimes_{G_F} G_E  \cong \Pic \: \overline{\Res X}$ of Galois modules. i.e. $\Pic \:\overline{\Res X}$ is the representation induced from the action of $G_F$ on 			$\Pic \Xbar$. \label{Picinduced}
	\end{enumerate}
	If in addition $X(F) \neq \emptyset$, then
	\begin{enumerate}[resume]
	    \item The map $\Res: \Pic X \to \Pic \Res X$ is an isomorphism. \label{Picisom}
	    \item The induced map on effective cones $\Res: \Eff(X) \to \Eff(\Res X)$ is an isomorphism. \label{Effisom}
	\end{enumerate}
\end{lemma}
\begin{proof}
    In what follows we identify $\overline{\Res X}$ with $\prod_{i=1}^d X^{\sigma_i}$ and $\overline{\Res L}$
	with $\boxtimes_{i=1}^d L^{\sigma_i}$ for any line bundle $L$ on $X$.
    As in the proof of Lemma \ref{lem:line_bundles}, we see that the K\"{u}nneth formula for coherent sheaves and flat base change imply
    that $h^1(\Res X, \mathcal{O}_{\Res X})=h^1(X, \mathcal{O}_{X})^d=0$, thus proving (\ref{h^1_structure_sheaf}).
    As $H^1(X, \mathcal{O}_X)=0$, it follows from \cite[Ex. III.12.6]{Har77} that the map
    \begin{align*}
        &\bigoplus_{i=1}^d \Pic X^{\sigma_i} \to \Pic \left(\prod_{i=1}^d X^{\sigma_i}\right)\\
        &(L_1^{\sigma_1},\ldots,L_d^{\sigma_d}) \mapsto \boxtimes_{i=1}^d L_i^{\sigma_i},
    \end{align*}
    is an isomorphism. This map is obviously a homomorphism of $G_F$-modules and thus proves (\ref{Picinduced}). As in the proof of Lemma \ref{lem:Pic0}, the fact that
	$X(F) \neq \emptyset$ implies (\ref{Picisom}). Finally note that $\Pic^0 X = 0$ by \cite[Thm.~8.4.1]{BLR90} as $H^1(X, \mathcal{O}_X)=0$.
	Therefore $\NS X = \Pic X$, and so (\ref{Effisom}) follows from (\ref{Picisom}), Lemma \ref{lem:line_bundles} and  Lemma \ref{lem:Pic0}
\end{proof}

\section{Adelically metrised line bundles} \label{Sec:adelic}
The aim of this section is to define the Weil restriction of an adelically metrised line bundle. We begin by recalling
various facts about height functions and adelically metrised line bundles, which can be found for example in \cite[Sec. 2]{CLT10a} or \cite[Sec. 2]{Pey03}.
Throughout this section $F$ is a number field.

\subsection{Definitions and basic properties}

\begin{definition}
	Let $(X,L)$ be a variety over $F$ together with a line bundle. For a place $v \in \Val(F)$,
    a $v$-adic metric on $L$ is a map which associates to every point $x_v \in X(F_v)$ a function $||\cdot||_v:L(x_v) \mapsto \RR_{\geq 0}$ on the fibre of $L$ above $x_v$ such that
	\begin{enumerate}
    		\item For all $\ell \in L(x_v)$, we have $||\ell||_v =0$ if and only if $\ell=0$.
    		\item For all $\lambda \in F_v$ and $\ell \in L(x_v)$, we have $||\lambda \ell||_v=|\lambda|_v ||\ell||_v$.
	    \item For any open subset $U \subset X$ and any local section $s \in \Gamma(U,L)$, the function given by $x_v \mapsto ||s(x_v)||_v$ is continuous in the $v$-adic topology.
	\end{enumerate}
\end{definition}
For any non-archimedean place $v$ of $F$, there is a natural way to associate 
a $v$-adic metric on $L$ to any model $(\mathscr{X},\mathscr{L})$ of $(X,L)$ over $\mathcal{O}_{F_v}$
(see \cite[Sec. 2.1.5]{CLT10a} or \cite[Ex. 2.2]{Pey03}).

\begin{definition}
	Let $(X,L)$ be a projective variety over $F$ together with a line bundle. An adelic metric on $L$ is a collection
	$||\cdot||=\{||\cdot||_v\}_{v \in \Val(F)}$ of $v$-adic metrics for each place $v \in \Val(F)$, such that all
	but finitely many of the $||\cdot||_v$ are defined by a single model of $(X,L)$ over $\mathcal{O}_F$.
	We denote the associated adelically metrised line bundle by $\LL=(L,||\cdot||)$.
\end{definition}

To any such pair $(X,\LL)$ one may associate a height function $H_{\LL}$. Namely,
for any $x \in X(F)$ define 
\begin{equation} \label{def:height}
	H_{\LL}(x) = \prod_{v \in \Val(F)}||s(x)||^{-1}_v,
\end{equation}
where $s$ is any local section of $L$ which is defined and non-zero at $x \in X(F)$. The fact that this definition is independent of $s$ follows from the product formula.

\begin{example}
	For each set of generating global sections $s_0,\ldots,s_n$ of $\mathcal{O}_{\PP^n}(1)$, 
	there is an adelic metric on $\mathcal{O}_{\PP^n}(1)$
such that for any local section $s$ of $\mathcal{O}_{\PP^n}(1)$ which is non-zero at $x_v \in \PP^n(F_v)$, the $v$-adic metric is given by
\begin{equation} \label{eqn:proj_metric}
	||s(x_v)||_v = \left( \max_{0 \leq i \leq n} \left|\frac{s_i(x_v)}{s(x_v)}\right|_v\right)^{-1}.
\end{equation}
If one takes $s_i=x_i$ for each $i=0,\ldots,n$, then the associated height function is exactly (\ref{eq:height}). One obtains other metrisations of $\mathcal{O}_{\PP^n}(1)$ by allowing
arbitrary $F_v$-vector space norms for any finite collection of places $v$ of $F$ in (\ref{eqn:proj_metric}),
instead of the usual maximum norm.
\end{example}

Many natural operations on line bundles can also be performed on adelically metrised 
line bundles. For example, given a morphism of projective varieties $f:Y \to X$
and an adelically metrised line bundle $\LL$ on $X$, there is a pull-back 
adelically metrised line bundle $f^*\LL$ on $Y$, for which we have an equality
$$H_{f^*\LL}(y) = H_{\LL}(f(y)),$$ of heights for all $y \in Y(F)$.
One may similarly define the dual and tensor product of adelically metrised line bundles
in a natural way, whose height functions satisfy the obvious relations (see \cite[Sec. 2.1.3]{CLT10a}).

We say that two adelically metrised line bundles $\LL_1$ and $\LL_2$ on a projective variety $X$ are \emph{isometric} if there exists an isomorphism of line bundles $\varphi:L_1 \to L_2$
and constants $\lambda_v \in \RR_{>0}$ for each $v \in \Val(F)$ such that $\prod_{v \in \Val(F)}\lambda_v =1$ and with the property that
for all $x_v \in X(F_v)$ and all local sections $s$ of $L_2$ defined at $x_v$ we have
$$||\varphi^* s (x_v)||_{1,v}=\lambda_v||s(x_v)||_{2,v}.$$
It is simple to see that isometric adelically metrised line bundles give rise to the same height function.
As an example of an isometry, note that since $X$ is projective any automorphism of a line bundle is given by multiplication by an element of $F^*$. It follows from the product formula that such a map is an isometry.

\subsection{The Weil restriction of an adelically metrised line bundle}
\label{Subsec:Weil_adelic_metric}
Throughout this section $E \subset F$ is an extension of number fields. Recall (\ref{def:Weil_Res_Line_bundle}) that given a projective variety $X$ over $F$ and
a line bundle $L$ on $X$, we defined $\Res L = \Norm{F}{E}(p^* L)$. In order to extend this definition to adelically metrised line bundles,
it suffices to define the norm of an adelically metrised line bundle.

\begin{example}
    \label{Ex:norms}
    Given a projective variety $Y$ over $E$ and an adelically metrised line bundle $\LL$ on $Y_F$, we shall now define the norm $\Norm{F}{E}(\LL)=(\Norm{F}{E}(L),||\cdot||)$ of $\LL$.
    For any place $v \in \Val(E)$ and any $y_v \in Y(E_v)$ there is a unique $v$-adic norm on $\Norm{F}{E}(L)$ such that
    \begin{equation} \label{def:norm_adelic_line_bundle}
    	||\Norm{F}{E} (s) (y_v)||_v = \prod_{w|v} ||s(y_v)||_w,
    \end{equation}
    for any local section $s$ of $L$ defined at $y_v$. Verifying that these define $v$-adic metrics
    is elementary. To see that these indeed come from a model for all but finitely many places,
    choose a finite set of places $S \subset \Val(E)$ containing all archimedean places and let $T \subset \Val(F)$  be the places of $F$ lying above those
	in  $S$. Let $(\mathscr{Y}',\mathscr{L})$ be a model of $(Y_F,L)$ over $\mathcal{O}_{F,T}$ and let $\mathscr{Y}$ be a model
	of $Y$ over $\mathcal{O}_{E,S}$. Taking $S$ sufficiently large, we may assume that the identity map $Y_F \to Y_F$ extends to an
	isomorphism $\mathscr{Y}' \cong \mathscr{Y}_{\mathcal{O}_{F,T}}$. Hence by taking the norm
	of $\mathscr{L}$ we obtain a model $(\mathscr{Y},\Norm{\mathcal{O}_{F,T}}{\mathcal{O}_{E,S}}(\mathscr{L}))$ of $(Y,\Norm{F}{E}(L))$
	over $\mathcal{O}_{E,S}$. In particular, we see that the above $v$-adic metrics (\ref{def:norm_adelic_line_bundle})
	do indeed come from a model for all but finitely many places.

    As for the height functions, one has $H_{\Norm{F}{E}\LL}(y)=H_{\LL}(y)$ for all $y \in Y(E)$.
    Indeed, choose a local section $s$ of $L$ which is defined and non-zero at $y$. Then by (\ref{def:norm_adelic_line_bundle}) we have
    \begin{align*} 
    	H_{\Norm{F}{E}(\LL)}(y)  = \prod_{v \in \Val(E)}||\Norm{F}{E}(s) (y)||_v^{-1}
    	 = \prod_{v \in \Val(E)} \prod_{w|v} ||s(y)||_w^{-1}
    	 = H_{\LL}(y),
    \end{align*} as required.
\end{example}

We therefore define an adelic metric on $\Res L$ by $\Res \LL = \Norm{F}{E}(p^* \LL)$.
The following lemma relates the height function $H_{\LL}$ to $H_{\Res \LL}$ and establishes (\ref{eqn:Weil_counting_functions}).

\begin{lemma} \label{lem:heights}
	Let $X$ be a projective variety over $F$ together with an adelically metrised line bundle $\LL$.
	Then we have $H_{\Res \LL}(x) = H_{\LL}(p(x))$
    for all $x \in \Res X(E)$. In particular,
    $$N(\LL,U,B)=N(\Res \LL,\Res U,B),$$
    for any open subset $U \subset X$ and any $B > 0$.
\end{lemma}
\begin{proof}
	It follows from immediately from the construction that
	$$H_{\Res \LL}(x)=H_{\Norm{F}{E}(p^*\LL)}(x)=H_{p^* \LL}(x)=H_{\LL}(p(x)),$$
	as required.
	The equality of counting functions follows from the fact that by definition, we have a bijection
	$p:\Res U (E) \to U(F)$
    induced by $p$.
\end{proof}

\section{Manin's conjectures}\label{Sec:Manin}
In this section we prove Theorem \ref{thm:Main}, Theorem \ref{thm:CI_Weil} and Theorem \ref{thm:counterexample}.
We also show that Peyre's refined conjecture
on the leading constant in the asymptotic formula is well-behaved under the Weil restriction (see Lemma \ref{lem:alpha_beta} and Theorem \ref{thm:Tamagawa}). Throughout this section $E \subset F$ is a finite extension of perfect fields of degree $d$ (assumed to be number fields 
from Section \ref{sec:Tamagawa} onwards).

\subsection{Proof of Theorem \ref{thm:Main}}
Let $(X,L)$ be a non-singular projective variety over $F$ together with a big line bundle such that $X(F) \neq \emptyset$.
In the light of Lemma~\ref{lem:heights}, to prove Theorem \ref{thm:Main} it suffices to show that the equalities
$a(L)=a(\Res L)$ and $b(L)=b(\Res L)$ hold. Here $a(L)$ is given by (\ref{def:a(L)}) and $b(L)$ is the codimension of the minimal face of $\partial\Eff(X)$ which contains $a(L)[D] + [\omega_X]$. Note that one needs to assume that $a(L)[D] + [\omega_X]$ belongs to the polyhedral part of $\partial\Eff(X)$ for $b(L)$ to be well-defined.
\begin{lemma}
    We have
    $$a(L)=a(\Res L).$$
    If moreover $[\omega_X]$ is not effective and $H^1(X,\mathcal{O}_X)=0$, then $b(L)$ is defined
    if and only if $b(\Res L)$ is defined.
    In which case we have
    $$b(L)=b(\Res L).$$
\end{lemma}
\begin{proof}
	Recall from Lemma \ref{lem:line_bundles} that the map $\Res : \Pic X \to \Pic \Res X$ preserves the canonical bundle and that a line bundle $L$ on $X$ is effective if and only if
    $\Res L$ is effective. Moreover as $\Pic^L X$ is a torsor for $\Pic^0 X$, it follows from Lemma \ref{lem:Pic0} that $\Pic^L X$ contains an effective line bundle if and only if $\Pic^{\Res L} \Res X$ does.
	Hence with respect to the induced injective linear map $\Res: \NS_{\RR} X \to \NS_{\RR} \Res X$, we see that $[L] \in \Eff(X)$ if and only if $\Res[L] \in \Eff(\Res X)$.
	It follows that
    \begin{align*}
        a(L)=&\inf \{ r \in \QQ : r[L] + [\omega_X] \in \Eff(X)\} \\
            =&\inf \{ r \in \QQ : \Res ( r[L] + [\omega_X]) \in \Eff(\Res X)\} \\
            =&\inf \{ r \in \QQ :   r[\Res L] + [\omega_{\Res X}] \in \Eff(\Res X)\} \\
            =&a(\Res L),
    \end{align*}
    as required. Now assume that $[\omega_X]$ is not effective and that $H^1(X,\mathcal{O}_X)=0$. Then by Lemma \ref{lem:Picard}, we see that we have an isomorphism of effective cones $\Res : \Eff(X) \to \Eff(\Res X)$ which preserves the canonical bundle. In particular as $a(L)=a(\Res L)$, we see that $a(L)[L] + [\omega_X]$ belongs to the polyhedral part of $\partial\Eff(X)$ if and only if the corresponding statement holds for $a(\Res L)[\Res L] + [\omega_{\Res X}]$, and moreover that $b(L)=b(\Res L)$.
\end{proof}
This completes the proof of Theorem \ref{thm:Main}.

\subsection{Peyre's constant}
In \cite{Pey95}, Peyre gave a refinement of Manin's original conjecture which predicts the form of the leading constant in the asymptotic formula (\ref{conj:2}) for Fano varieties. Namely, after fixing a choice of adelic metric on the anticanonical line bundle $\omega_X^{-1}$, he conjectured that the leading constant $c_{\omega_X^{-1}}$
should satisfy
$$c_{\omega_X^{-1}}=\alpha(X)\beta(X)\tau(X).$$
Here $\alpha(X)$ is defined to be
$$\alpha(X)= \frac{1}{(\rho - 1)!}\int_{\Eff(X)^{\vee}} e^{-\langle \omega_X^{-1},x\rangle} \mathrm{d}x,$$
where $\rho=\rank \Pic X$ and $\mathrm{d}x$ is the Haar measure on the dual vector space $(\Pic X \otimes_{\ZZ} \RR)^{\vee}$ normalised
so that $(\Pic X)^{\vee}$ has covolume $1$. Also $\beta(X)=\# H^1(G_F, \Pic \Xbar)$ and $\tau(X)$ is the ``Tamagawa number" of $X$ with respect to the choice of adelic metric on $\omega_X^{-1}$.
The main result of this section is that this refined conjecture is compatible with the Weil restriction, i.e. we have
an equality $c_{\omega_X^{-1}} = c_{\omega_{\Res X}^{-1}}$.
We begin with $\alpha(X)$ and $\beta(X)$.
\begin{lemma}\label{lem:alpha_beta}
    Let $X$ be a smooth projective variety over $F$ such that we have  $H^1(X,\mathcal{O}_X)=0$, $X(F)\neq \emptyset$
    and such that $\Pic X$ is a free abelian group of finite rank.
    Then
    $$\beta(X)=\beta(\Res X).$$
    If $\omega_X^{-1}$ is big then
    $$\alpha(X)=\alpha(\Res X).$$
\end{lemma}
\begin{proof}
    By Lemma \ref{lem:Picard} we know that $\Pic \overline{\Res X}$ is an induced representation of $\Pic \Xbar$. 
    Therefore Shapiro's lemma \cite[Prop.~1.6.3]{NSW00} implies that we have an isomorphism 
    $H^1(G_F, \Pic \Xbar) \cong H^1(G_E, \Pic \overline{\Res X})$, hence  $\beta(X)=\beta(\Res X)$.

    Next assume that $\omega_X^{-1}$ is big (this assumption is needed to make sure that $\alpha(X)$ is well-defined).
    By Lemma \ref{lem:line_bundles} and Lemma \ref{lem:Picard} we have an 	
    isomorphism $\Res : \Pic X \otimes_\ZZ \RR \to \Pic \Res X \otimes_\ZZ \RR $  which preserves the canonical bundle
	and induces an isomorphism of the Picard lattices and effective cones. As $\alpha(X)$ is defined purely in terms of this data and its dual, we see that $\alpha(X)=\alpha(\Res X).$
\end{proof}

\subsubsection{Tamagawa numbers} \label{sec:Tamagawa}
We next address the Tamagawa numbers, so we assume that $E \subset F$ are number fields.
Weil \cite{Wei82} was the first to define Tamagawa numbers of \emph{linear algebraic groups}, and he also showed 
\cite[Thm.~2.3.2]{Wei82} that they are preserved under the Weil restriction for finite separable extensions of global fields.
Weil's proof was however lacking in certain details, and a complete proof for all linear algebraic groups, including the 
non-separable case, was given by Oesterl\'{e} \cite[Thm.~II.1.3]{Oes84}.

Throughout this section $X$ is a smooth projective variety of dimension $n$ over $F$ such that
\begin{equation} \label{eqn:Tamagawa_conditions}
    \begin{array}{ll}
	\Pic \Xbar \text{ is a free abelian group of finite rank}.\\
	H^1(X,\mathcal{O}_X)=H^2(X,\mathcal{O}_X)=0.
    \end{array}
\end{equation}
We also fix a choice of adelic metric on the canonical line bundle $\omega_X$. For example $X$ could be a Fano variety with
the adelic metric coming from a choice of anticanonical embedding. We now recall the definition of the Tamagawa measure on 
$X (\Adele_F)$ associated to this choice of adelic metric. Such measures were originally defined by
Peyre \cite{Pey95} in the case where $X$ is Fano, however his construction also works in the slightly more general setting of 
(\ref{eqn:Tamagawa_conditions}) (see \cite{CLT10a}). Note that we do not impose any positivity conditions
on $\omega_X$.

For inspiration with the definition to come, let $K$ be a local field and $V$ a $K$-vector space  of dimension $n$. Then a choice of norm $|| \cdot ||$ on $\det(V)=\bigwedge^n V$ determines a measure on $V$.
Indeed, choosing an isomorphism $\phi:V \cong K^n$ we have the measure
$$\frac{\phi^*(|\mathrm{d}x_{1}|\cdots|\mathrm{d}x_n|)}
{||\phi^*(e_{1}\wedge\ldots\wedge e_{n})||},$$
on $V$, where $e_1,\ldots,e_n$ are the standard basis vectors on $K^n$ and
$|\mathrm{d}x_{1}|\cdots|\mathrm{d}x_n|$ denotes the product of the Haar
measures on $K^n$. It is easy to see that this is independent of the choice of $\phi$. To define measures on varieties
over local fields, we essentially apply this
construction to the cotangent space of each point on the variety.
For each place $v$ of $F$, choose a differential form $\omega_v$ of top degree defined on some open subset $U_v \subset X(F_v)$.
Then in a choice of local coordinates $x_{v,1},\ldots,x_{v,n}$ we may write $\omega_v$ as
$$\omega_v = f(x_{v,1},\ldots,x_{v,n})\mathrm{d}x_{v,1}\wedge\ldots\wedge\mathrm{d}x_{v,n}.$$
We define the measure $|\omega_v|_v$ associated to $\omega_v$ to be
$$|\omega_v|_v = |f(x_{v,1},\ldots,x_{v,n})_v|_v|\mathrm{d}x_{v,1}|_v\cdots|\mathrm{d}x_{v,n}|_v.$$
This measure is independent of the choice of local coordinates, however it depends on the choice of $\omega_v$. We therefore
consider instead the measure on $U_v$ given by
$|\omega_v|_v/||\omega_v||_v,$
which is independent of $\omega_v$. By gluing these measures, we obtain a measure $\tau_{X,v}$ on $X(F_v)$.

The product of these measures does not converge in general, so we need to introduce convergence factors to get a measure on
$X(\Adele_F)$. Since $\Pic \Xbar$ is a free abelian group of finite rank and moreover is a $G_F$-module, we may define the corresponding
Artin L-function $L(s,\Pic \Xbar)$ as a product of local factors $L_v(s,\Pic \Xbar)$ for each finite place $v \in \Val(F)$ 
(see e.g. \cite[Sec. 5.13]{IK04}). This L-function is holomorphic on $\re s >1$ and admits a meromorphic continuation to 
$\CC$ with a pole of order $\rho=\rank \Pic X$ at $s=1$. For each place $v \in \Val(F)$ we define
$$ \lambda_v = \left \{
	\begin{array}{ll}
		L_v(1,\Pic \Xbar),& \quad v \text{ non-archimedean}, \\
		1,& \quad v \text{ archimedean}. \\
	\end{array}\right.$$	
The condition $H^1(X,\mathcal{O}_X)=H^2(X,\mathcal{O}_X)=0$ implies (see \cite[Thm.~1.1.1]{CLT10a}) that these are a family of ``convergence factors'', i.e.
the measure
\begin{equation} \label{def:Tamagawa}
	\tau_X = \mu_F^{-n}\lim_{s \to 1}( (s-1)^{\rho} L(s,\Pic \Xbar) ) \prod_{v \in \Val(F)} \lambda_v^{-1} \tau_{X,v},
\end{equation}
is a well-defined measure on $X(\Adele_F)$, called the \emph{Tamagawa measure} of $X$.
Here $\mu_F$ denotes the volume of $\Adele_F/F$ with respect to our choice of Haar measure
(Peyre \cite[Def.~2.1]{Pey95} includes instead a discriminant factor due to his specific choice of Haar measure). 
We define the \emph{Tamagawa number} $\tau(X)$ of $X$ to be
$$\tau(X)=\tau_X(\overline{X(F)}),$$
where $\overline{X (F)}$ denotes the closure of $X (F)$ in $X (\Adele_F)$ with respect to the adelic topology. 
This construction depends on the choice of the adelic metric on the canonical line bundle $\omega_X$, but is independent
of the choice of Haar measure on $\Adele_F$.

We now consider the associated Tamagawa measure on $\Res X$. In order to get an adelic metric on $\omega_{\Res X}$, we need to choose an isomorphism of line bundles
$\phi:\omega_{\Res X} \to \Res \omega_X$. As $\Res \omega_X$ comes equipped with an adelic metric coming from $\omega_X$, by transport of structures
we obtain an adelic metric on $\omega_{\Res X}$. Firstly note that this adelic metric is independent of the choice of $\phi$, up to isometry.
Indeed, the choice of the isomorphism $\phi$ depends only on the choice of a non-zero global section $\varphi$ of $\omega_{\Res X} \otimes  \Res \omega_X^{-1}$.
As $\Res X$ is projective, any other choice of global section must differ from $\varphi$ by a non-zero scalar, and hence defines an isometric adelic
metric on $\omega_{\Res X}$.
Next note that we deduce from (\ref{eqn:Tamagawa_conditions}) and Lemma \ref{lem:Picard}
that $\Pic \overline{\Res X}$ is free of finite rank and that $H^1(\Res X, \mathcal{O}_{\Res X})=0$. Also using the K\"{u}nneth formula as in the proof of
Lemma \ref{lem:Picard}, we find that $H^2(\Res X, \mathcal{O}_{\Res X})=0$. Hence $\Res X$ also
satisfies the conditions (\ref{eqn:Tamagawa_conditions}) and we have constructed a well-defined Tamagawa measure $\tau_{\Res X}$.
As isometric adelic metrics clearly give rise to the same Tamagawa measure, we see that $\tau_{\Res X}$ is independent of the choice of $\phi$.

Note that there is another natural way to define a measure on $\Res X(\Adele_E)$. 
Namely, we may simply pull-back the Tamagawa measure $\tau_X$ on $X(\Adele_F)$
via the homeomorphism $p:\Res X(\Adele_E) \to X(\Adele_F)$. Our main result of 
this section is that these two constructions coincide.

\begin{theorem} \label{thm:Tamagawa}
    Suppose that $X(F) \neq \emptyset$. Then we have $p^*\tau_X = \tau_{\Res X}$, i.e. the map
	$$p: \Res X (\Adele_E) \to X (\Adele_F),$$
	is an isomorphism of  measure spaces. In particular there is an equality
	$\tau(X)=\tau(\Res X)$
	of Tamagawa numbers.
\end{theorem}
We begin the proof of the theorem by considering the L-functions and convergence factors.

\begin{lemma} \label{lem:Lfunctions}
	There is an equality
	$$L(s,\Pic \Xbar) = L(s,\Pic \overline{\Res X}),$$
	of L-functions and for any place $v \in \Val(E)$ an equality
	$ \lambda_v = \prod_{w|v} \lambda_w$
	of convergence factors.
\end{lemma}
\begin{proof}
	Lemma \ref{lem:Picard} implies that $\Pic \overline{\Res X}$ is the induced representation of $ \Pic \Xbar$ with respect to the field extension $E \subset F$.
	This gives the equality of L-functions and an equality $$ L_v(s,\Pic \Xbar) = \prod_{w|v} L_w(s,\Pic \overline{\Res X})$$ of local factors for
	each non-archimedean place $v \in \Val(E)$ (see e.g. \cite[Sec. 5.13]{IK04}). This completes the proof of the lemma.
\end{proof}

Next we consider the local measures. Recall that we have chosen a non-zero global section $\varphi$ 
of $\omega_{\Res X} \otimes  \Res \omega_X^{-1}$, which induces an isometry $\phi:\omega_{\Res X} \to \Res \omega_X$
of line bundles, and also that we have a homeomorphism $p_v: \Res X(E_v) \to \prod_{w|v} X(F_w)$ for any place $v \in \Val(E)$.
\begin{lemma} \label{lem:local_measures}
	For each place $v \in \Val(E)$ there exists a constant $A_v$ depending on $\varphi$ and $v$, such that
	$$p_v^* \left( \prod_{w|v} \tau_{X,w}\right) = A_v \cdot \tau_{\Res X,v}.$$
\end{lemma}
\begin{proof}
    Let $\omega$ be a local algebraic differential form of top degree on $X$. To prove the lemma, it suffices to show that
    there exists a constant $A_v$, depending on $\varphi$ and $v$,  such that
    $p_v^* \prod_{w|v}|\omega|_w/||\omega|_w = A_v|\phi^* \Res \omega|_v/||\phi^* \Res \omega||_v$.
    Note that such a constant  is necessarily independent of $\omega$; indeed these measures are independent of $\omega$.
    Also as by definition we have an equality $ \prod_{w|v}||\omega||_w = ||\phi^* \Res \omega||_v$,
    we only need to show that $p_v^* \prod_{w|v}|\omega|_w = A_v|\phi^* \Res \omega|_v$.
    To do this we work locally near each point $x_v \in \Res X(E_v)$, following a similar approach to Oesterl\'{e} (cf. \cite[Lem. II.5.2]{Oes84} 
    and \cite[Ex. II.4.3]{Oes84}).

    Note that by the definition of the Weil restriction, for any $E$-algebra $R$ we have a canonical bijection
    $\Hom(\Spec R[\varepsilon]/(\varepsilon^2),\Res X) \to \Hom(\Spec (R \otimes_E F) [\varepsilon]/(\varepsilon^2),X)$.
    In particular, we may canonically identify the cotangent space $T_{x_v}^*$  of a point $x_v \in \Res X(E_v)$ with the space $\prod_{w|v} T^*_{x_w}$ considered as an $E_v$-vector space, where we write $p_v(x_v)=(x_w)_{w|v}$.
    If we let $L=\omega_{\Res X} \otimes  \Res \omega_X^{-1}$, then under this correspondence
    we have isomorphisms 
    \begin{equation} \label{eqn:cotangent_isometry}
    \begin{aligned} 
		\textstyle{\det_{E_v}}T_{x_v}^* &\cong L(x_v) \bigotimes \Norm{F \otimes E_v}{E_v} \textstyle{\det_{F \otimes E_v}} T^*_{p_v(x_v)} \\
		&\cong L(x_v) \bigotimes_{w|v} \Norm{F_w}{E_v} \textstyle{\det_{F_w}} T^*_{x_w},
    \end{aligned}
    \end{equation}
    where $\phi^* \Res \omega(x_v)$ is identified with $\varphi(x_v) \otimes_{w|v} \Norm{F_w}{E_v} \omega(x_w)$.

    Next, for each $w|v$ choose an isomorphism $f_w:T^*_{x_w} \to F_w^n$ such that 
    we have $\det f_w(\omega(x_w)) = e_1\wedge\ldots\wedge e_n$.
    It follows that the map $f_w$ is measure preserving and moreover we may identify $T_{x_v}^*$ with $V=\prod_{w|v}F_w^n$ 
	considered as an $E_v$-vector space. As there are canonical isomorphisms $\Norm{F_w}{E_v} \det_{F_w}F_w^n \cong E_v$, 
	the isomorphism (\ref{eqn:cotangent_isometry}) simply becomes $\det_{E_v} V \cong L(x_v)$. In particular we see that the 
	two different measures on $V$,  being both Haar measures on the same locally compact topological group, differ
    by a constant $A_v(x_v)$, which depends only on $\varphi(x_v)$ and the field $E_v$.
    But as $L$ is isomorphic to the trivial line bundle, it has constant fibres and in particular $\varphi(x_v)$, and hence $A_v(x_v)$, 
    is in fact independent of $x_v$. This completes the proof of the lemma.

\end{proof}
In order to use Lemma \ref{lem:local_measures} to deduce a global result, we shall need the following.

\begin{lemma}\label{lem:A_v=1}
	For all but finitely many $v \in \Val(E)$ we have $A_v=1$.
\end{lemma}
\begin{proof}
	To prove the lemma, it suffices to compare the measures of two measurable sets inside $\Res X(E_v)$ and  $\prod_{w|v}X(F_w)$
	which are identified by $p_v$. In particular it is sufficient to show that for all but finitely many $v \in \Val(E)$ we have
	\begin{equation}\label{eqn:points_mod_p}
		\tau_{\Res X,v}(\Res X(E_v))=\prod_{w|v}\tau_{X,w}(X(F_w)).
	\end{equation}
	First choose a model $\mathscr{X}$ of $X$ over $\mathcal{O}_F$. In which case $\RRes{\mathcal{O}_F}{\mathcal{O}_E}\mathscr{X}$ is also a model
	of $\Res X$ over $\mathcal{O}_E$. It then follows from \cite[Sec. 2.4.1]{CLT10a} that for all but finitely many non-archimedean places $v \in \Val(E)$
	we have
	\begin{align*}
		\tau_{\Res X,v}(\Res X(E_v)) &= \frac{\#\RRes{\mathcal{O}_F}{\mathcal{O}_E}\mathscr{X}(\mathcal{O}_E/\pp_v)}{N(\pp_v)^{dn}},  \\
		\prod_{w|v}\tau_{X,w}(X(F_w)) &= \prod_{w|v}\frac{\#\mathscr{X}(\mathcal{O}_F/\pp_w)}{N(\pp_w)^{n}},
	\end{align*}
	where $\pp_v$ and $\pp_w$ denote the prime ideals corresponding to the places $v$ and $w$ respectively. However, by the definition
	of the Weil restriction we have an equality of sets $\RRes{\mathcal{O}_F}{\mathcal{O}_E}\mathscr{X}(\mathcal{O}_E/\pp_v)=
	\prod_{w|v}\mathscr{X}(\mathcal{O}_F/\pp_w)$. Also, as $N(\pp_v)^{d}=\prod_{w|v}N(\pp_w)$ for unramified primes $\pp_v$ (in particular for all but finitely many $v$), this shows that
	(\ref{eqn:points_mod_p}) holds for all but finitely many places and completes the proof of the lemma.
\end{proof}

Hence from Lemma \ref{lem:Lfunctions}, Lemma \ref{lem:local_measures}, Lemma \ref{lem:A_v=1} and the definition (\ref{def:Tamagawa}) of the Tamagawa measures, we see that
\begin{equation} \label{eq:A}
	\mu_F^{n}\cdot p^*\tau_X = A\cdot \mu_E^{nd}\cdot \tau_{\Res X},
\end{equation}
where $A=\prod_{v \in \Val(E)} A_v$.
Our next aim is to calculate $A$, which we may do by working locally near a single rational point.
The following lemma will assist with this calculation.

\begin{lemma} \label{lem:cotangent}
	Let $x \in X(F)$ and consider the adelic cotangent space $T^*_x \otimes_F \Adele_F$ equipped with the
	measure induced by the adelic metric on $\omega_X$. Then
	$$\vol\left((T^*_x \otimes_F \Adele_F)/T^*_x\right) = \mu_F^{n}\cdot H_{\omega_X}(x).$$
\end{lemma}
\begin{proof}
	Choose algebraic local coordinates $x_1,\ldots,x_n$ near $x$ defined over $F$ and let 
    $\omega=\mathrm{d}x_1\wedge\ldots\wedge\mathrm{d}x_n$. These local coordinates give an isomorphism
	$T^*_x \cong F^n$ with respect to which we have $|\omega|_v=|\mathrm{d}x_{v,1}|_v\cdots|\mathrm{d}x_{v,n}|_v$ for each place $v \in \Val(F)$.
	Therefore the measure $\prod_v|\omega|_v$ on $T^*_x \otimes_F \Adele_F$ is identified with the product measure on $\Adele_F^n$ (note that
	we do not require convergence factors). Also, by \eqref{def:height} we have $\prod_v||\omega||^{-1}_v=H_{\omega_X}(x)$. 
    Therefore the measure $\prod_v|\omega|_v/||\omega||_v^{-1}$ on $T^*_x \otimes_F \Adele_F$
	is identified with $H_{\omega_X}(x)$ times the product measure on $\Adele_F^n$, and the result follows.
\end{proof}

On choosing a rational point $x \in \Res X(E)$ we have an induced map $T^*_{p(x)} \to T^*_x$ 
of $E$-vector spaces which induces a homeomorphism
$T^*_{p(x)} \otimes_F \Adele_F \to T^*_x \otimes_E \Adele_E$. Therefore by the definition of $A$ we have
$$\vol((T^*_{p(x)} \otimes_F \Adele_F)/T^*_{p(x)})=A\cdot \vol((T^*_x \otimes_E \Adele_E)/T^*_x),$$
with respect to the associated measures.
However, on applying Lemma \ref{lem:cotangent} we see that
\begin{align*}
	\vol((T^*_{p(x)} \otimes_F \Adele_F)/T^*_{p(x)})&=\mu_F^{n}\cdot H_{\omega_{X}}(p(x)), \\
	\vol((T^*_x \otimes_E \Adele_E)/T^*_x)&=\mu_E^{nd}\cdot H_{\Res \omega_{X}}(x).
\end{align*}
As $H_{\Res \omega_{X}}(x) = H_{\omega_{X}}(p(x))$ by Lemma \ref{lem:heights}, 
we obtain $A=\mu_F^{n}/ \mu_E^{nd}$. Combining this calculation with (\ref{eq:A}) proves that $p^*\tau_X = \tau_{\Res X}$.

As for the equality of Tamagawa numbers, we note that  $p: \Res X (\Adele_E) \to X (\Adele_F)$ is continuous and restricts to a bijection
$p: \Res X (E) \to X (F)$. It therefore induces a measure preserving bijection $p:\overline{\Res X (E) } \to \overline{X (F)}$
and hence $\tau(X)=\tau(\Res X)$. This completes the proof of Theorem \ref{thm:Tamagawa}.

\subsection{Complete intersections} \label{Sec:CI}
We now prove Theorem \ref{thm:CI_Weil}. As noted in the introduction,
we are able to handle more general height functions than simply the height function (\ref{eq:height}).
By Theorem \ref{thm:Main}, Lemma \ref{lem:alpha_beta} and Theorem \ref{thm:Tamagawa},
it suffices to show the following result.

\begin{theorem}\label{thm:CI}
	Let $F$ be a number field and let $X \subset \PP^n$ be a non-singular complete intersection of $m$
	hypersurfaces over $F$ each of the same degree $r$. Suppose that $$n + 1 - \dim X^* > m(m+1)(r-1)2^{r-1}$$
	and $X(\Adele_F) \neq \emptyset$. For each archimedean place $v$ of $F$, choose an arbitrary $F_v$-vector space norm $||\cdot||_v$
	on $F_v^{n+1}$ and for each non-archimedean place $v$ let $||\cdot||_v$ be the usual maximum norm. 
	Let $\LL=(\mathcal{O}_X(1),||\cdot||)$ denote the associated adelically metrised line bundle.
	Then 
	$$N(\LL,X,B) \sim c B^{n+1 - mr}, \quad \text{as } B \to \infty,$$
	where $c=c(\LL,X)>0$ agrees with Peyre's prediction.
\end{theorem}
Note that this result does indeed confirm (\ref{conj:2}), as $\omega_X^{-1} \cong \mathcal{O}_X(n+1-mr)$
for such complete intersections (see \cite[Prop.~4]{FMT89}).
In the statement of the theorem $X^*$ denotes the ``Birch singular locus'' of X (see \cite{Ski97}).
Even though $X$ is non-singular this may be non-empty. We have the upper bound
$\dim X^* \leq m$, which was used to simplify the statement of Theorem \ref{thm:CI_Weil}.
As explained in the introduction, our proof hinges on the results of \cite{Ski97}.
Peyre describes in great detail in \cite[Sec. 5]{Pey95} the relationship between the circle
method and Manin's conjecture, and we follow his ideas closely. Peyre works over general
number fields and only specialises to the case where $F=\mathbb{Q}$ at the end of his discussion,
presumably because Skinner's result came after Peyre's paper. 

We begin with some notation.
For each ideal $\aaa$ of $\mathcal{O}_F$, we denote its norm by $\mathcal{N}(\aaa)$
and we denote by $\aaa_v = \aaa \otimes_{\mathcal{O}_F} \mathcal{O}_{F_v}$
for each non-archimedean place $v$ of $F$.
We also choose representatives $\cc_1,\ldots,\cc_h$ for the class group of $\mathcal{O}_F$
and let $\cc=\{\cc_1,\ldots,\cc_h\}$.
For each non-zero ideal $\aaa$, we choose some $\lambda_\aaa \in F^*$ such that 
$\aaa = \lambda_{\aaa}\cc_{\aaa}$
for some $\cc_{\aaa}\in \cc$.
By changing the $\cc_i$ if necessary, we may assume that $\lambda_\aaa \in \mathcal{O}_F$
for each $\aaa$. The first step of the proof is to lift the counting problem to one of counting integral 
points on the affine cone of $X$. Choosing the coefficients of the equations of $X$ to lie inside $\mathcal{O}_F$,
we obtain a model for $X$ over $\mathcal{O}_F$. We denote the affine cone of this model by $W \subset \mathbb{A}^{n+1}_{\mathcal{O}_F}$.
In what follows we shall identity $W(F)$ with its image in $F_\infty^{n+1}$.
For any bounded subset $O \subset F_\infty^{n+1}$ and any
non-zero ideal $\aaa$ of $\mathcal{O}_F$ we define
$$N(O,\aaa,B) = \#\{x \in W(\aaa) \cap B^{1/d}O: x \neq 0\}.$$
Here we write $W(\aaa) = W(F) \cap \aaa^{n+1}$ and $d=[F:\QQ]$.
Denote by $\omega_L=\mu_F^{-\dim X}\prod_{v \in \Val(F)} \omega_{L,v}$
the Leray form on $W(\Adele_F)$ (see \cite[Sec. 5.2]{Pey95}).
The following lemma is an application of the main result of \cite{Ski97}.

\begin{lemma}\label{lem:Balls}
	Let $O \subset F_\infty^{n+1}$
	be a bounded open subset such that the boundary of $O \cap W(F_\infty)$ has zero
	measure with respect to $\prod_{v \mid \infty}\omega_{L,v}$.
	Then 
	$$N(O,\aaa,B) = c(O,\aaa) B^{n+1 - mr} 
	+ o_{O,\aaa}(B^{n+1 - mr}), \quad \text{ as } B \to \infty,$$
	where $$c(O,\aaa) = \int_{O \times \prod_{v \nmid \infty} \aaa_v} \omega_L.$$
\end{lemma}
\begin{proof}
	Choose an integral basis $\omega_1,\ldots,\omega_d$ for $\aaa$.
	We shall work with the ``boxes'' inside $F_\infty^{n+1}$ given by the translates of
	$$\mathscr{B}_k=\{(r_0,\ldots,r_{n}) \in F_\infty^{n+1} : -(1/2)^k \leq r_{ij} < (1/2)^k, \quad 0 \leq i \leq n, 1 \leq j \leq d\},$$
	where $k \in \NN$ and we write $r_i = \omega_1 r_{i1} + \cdots + \omega_d r_{id}.$
	Let $O_{k,-}$ be the union of all non-overlapping translates of $\mathscr{B}_k$ strictly contained in $O$ 
	and let $O_{k,+}$ be the union of 
	all non-overlapping translates of $\mathscr{B}_k$ which strictly contain $O$.
	The assumption that $O$ is bounded implies that there are only finitely many such boxes.
	Clearly
	\begin{equation} \label{eqn:O}
		N(O_{k,-},\aaa,B)\leq N(O,\aaa,B) \leq N(O_{k,+},\aaa,B).
	\end{equation}
	Let $$c(O_{k,\pm},\aaa) = \lim_{B \to \infty} \frac{N(O_{k,\pm},\aaa,B)}{B^{n+1-mr}}.$$
	Skinner's result \cite{Ski97} implies that $c(O_{k,\pm},\aaa)$ exists, is finite and non-zero.
	The constants $c(O_{k,\pm},\aaa)$ are given by the product of the usual singular series and singular integral,
	which by a standard argument (see \cite[Sec. 5]{Pey95}) may be written as
	$$c(O_{k,\pm},\aaa) = \int_{O_{k,\pm} \times \prod_{v \nmid \infty} \aaa_v} \omega_L.$$
	Let
	\begin{equation} \label{eqn:c_O}
		c(O,\aaa) = \liminf_{B \to \infty} \frac{N(O,\aaa,B)}{B^{n+1-mr}}.
	\end{equation}
	Using (\ref{eqn:O})	we obtain
	$$
		c(O_{k,-},\aaa) \leq c(O,\aaa) \leq c(O_{k,+},\aaa).
	$$
	In particular $c(O,\aaa)$ is finite and non-zero.
	Our assumptions on $O$ imply that $$\lim_{k \to \infty} (c(O_{k,+},\aaa) - c(O_{k,-},\aaa)) = 0.$$
	The dominated convergence theorem therefore implies that
	$$c(O,\aaa) = \int_{O \times \prod_{v \nmid \infty} \aaa_v} \omega_L.$$
	The result is proved on applying the same argument again with the $\liminf$ in (\ref{eqn:c_O}) replaced by a $\limsup$.
\end{proof}
We now define the \Mob function multiplicatively on ideals of $\mathcal{O}_F$ via 
$$\mu(\mathfrak{p}) = -1, \qquad \mu(\mathfrak{p}^\nu) = 0, \quad \nu > 1,$$
for any non-zero prime ideal $\mathfrak{p}$ of $\mathcal{O}_F$.
Applying \Mob inversion (see \cite[Prop.~5.4.1]{Pey95} or \cite[Prop.~2.4.2]{Pey01}), we obtain
\begin{equation} \label{eqn:Mobius}
	N(\LL,X,B) = \frac{1}{w}\sum_{i = 1}^h \sum_{\aaa} \mu(\aaa)N(O_{\cc_{i}\aaa},\cc_i \aaa,\mathcal{N}(\cc_i)B),
\end{equation}
where the sum is over all non-zero ideals $\aaa$ of $\mathcal{O}_F$ and $w$ denotes the number of roots of unity in $\mathcal{O}_F^*$.
Here
$$O_{\aaa} = (\lambda_{\aaa} \Delta_F) \cap \left\{ x \in W(F_\infty): \prod_{v \mid \infty} ||(x_{0,v},\ldots,x_{n,v})||_v < 1\right\},$$
where $\Delta_F$ denotes the fundamental domain for the action of $\mathcal{O}_F^*$ on $F_\infty^{n+1}$
as constructed by Schanuel (see \cite{Sch79} or \cite[Sec. 5.1]{Pey95}).
Note that in \cite[Prop.~5.4.1]{Pey95}, Peyre uses the same fundamental domain
for each ideal $\aaa$, whereas here it is more convenient for us to allow different fundamental domains for 
different ideals (the same proof also works in this case).
By \cite[Prop.~2]{Sch79}, the set $O_\aaa$ is bounded. Moreover, its boundary clearly
has zero measure with respect to $\prod_{v \mid \infty}\omega_{L,v}$ (it is contained in 
$\{x \in W(F_\infty): \prod_{v \mid \infty} ||(x_{0,v},\ldots,x_{n,v})||_v = 1\}$). Hence
$O_\aaa$ satisfies the conditions of Lemma \ref{lem:Balls},
though there is a slight problem for the summation \eqref{eqn:Mobius}
as the error term in Lemma \ref{lem:Balls} is not uniform with respect $\aaa$.
We shall get around this by using a trick,
which was also used by Schanuel in \cite{Sch79}.
Namely, we have chosen the regions $O_\aaa$ in such a manner that 
$N(O_\aaa,\aaa,B) = N(O,\cc_\aaa,B/\mathcal{N}(\lambda_{\aaa}))$
where $O = O_{(1)}$.
Therefore, applying Lemma \ref{lem:Balls} to (\ref{eqn:Mobius}) we obtain
\begin{align*}
	N(\LL,X,B) =  &B^{n+1 - mr} \cdot \frac{1}{w}\sum_{i = 1}^h 
	\sum_{\aaa} \mu(\aaa)c(O_{\cc_i\aaa},\cc_i\aaa) \mathcal{N}(\cc_i)^{n+1 - mr} \\
	& \qquad +  o_{O,\cc}\left(\sum_{\aaa}\left(\frac{B}{\mathcal{N}(\aaa)}\right)^{n+1-mr}\right).
\end{align*}
By \cite[Prop.~5.4.1]{Pey95}, the sum in the main term converges and agrees with Peyre's prediction.
The convergence of the sum in the error term follows from the fact that the
Dedekind zeta function $\zeta_F(s)= \sum_{\aaa} \mathcal{N}(\aaa)^{-s}$ is absolutely convergent for $\re s >1$.
This completes the proof of Theorem \ref{thm:CI}, hence also the proof of Theorem \ref{thm:CI_Weil}.

\subsection{Counterexamples to Manin's conjecture}
\label{Sec:counterexample}
We now finish off our paper by proving Theorem \ref{thm:counterexample}.
We begin by recalling a special case of the counterexamples constructed by Batyrev and Tschinkel \cite{BT96}.
Let $X$ be the hypersurface in $\PP^3 \times \PP^3$  defined by the equation
\begin{equation} \label{eqn:BT}
	x_0y_0^3 + x_1y_1^3 + x_2y_2^3 + x_3y_3^3=0.
\end{equation}
Then Batyrev and Tschinkel \cite[Thm.~3.1]{BT96} have show that for any number field $F$ containing $\QQ(\sqrt{-3})$, any choice of adelic metric on $\omega_{X_F}^{-1}$ and any non-empty open subset $U \subset X_{F}$,
there exists a constant $c>0$ such that
$$N(\omega_{X_F}^{-1},U,B) \geq cB(\log B)^{3},$$
for any $B >0$. As $X$ is a smooth Fano variety with $\Pic X \cong \ZZ^2$, this
provides a counterexample to Manin's conjecture  (\ref{conj:2})  over such fields $F$.

Our counterexamples will be Weil restrictions of (\ref{eqn:BT}). To begin with, we need lower bounds on the number
of rational points of bounded height on the Weil restrictions of certain cubic surfaces. In what follows, we use various standard facts
about del Pezzo surfaces which can be found for example in \cite{Man86} or \cite[Ch. V.4]{Har77}. Recall also that we say
that a del Pezzo surface $S$ over a field $F$ is \emph{split} if the natural map $\Pic S \to \Pic \overline{S}$
is an isomorphism. In particular, a smooth cubic surface is split if and only if all of its lines are defined over the ground field.

\begin{lemma}\label{lem:split_cubic}
    Let $E \subset F$ be a quadratic extension of number fields,
    let $S$ be a smooth split cubic surface over $F$ and let $S'=\Res S$.
    Let $E \subset F'$ be a finite field extension and choose an adelic metric on $\omega_{S'_{F'}}^{-1}$. Let
    $U \subset S'_{F'}$ be an open subset and let $B>0$. Then there exists a constant $c>0$
    such that
	$$N(\omega_{S'_{F'}}^{-1},U,B) \geq cB(\log B)^{7},$$
    if $F \subset F'$ and
    $$N(\omega_{S'_{F'}}^{-1},U,B) \geq cB(\log B)^{3},$$
    otherwise.
\end{lemma}
\begin{proof}
    As all the lines in $S$ are defined over $F$, we may contract any three non-intersecting lines $L_1,L_2$ and
    $L_3$ to obtain a morphism $\pi:S \to Y$ defined over $F$, where $Y$ is a split del Pezzo surface of degree six.
    Let $Y'=\Res Y$ and let  $\pi':S' \to Y'$ be the induced map. Note that since $Y$ is toric by \cite[Thm.~30.3.1]{Man86}, we deduce that $Y'$ is 			
	also toric under the torus $T=\Res \Gm^2$. Choose an adelic metric on $\omega_{Y'_{F'}}^{-1}$.

    It follows from \cite[Prop.~V. 3.3]{Har77} that $K_S -\pi^*(K_{Y})  \sim  L_1 + L_2 + L_3 \geq 0$, where $K_S$ denotes a canonical divisor of $S$.
    Hence, we see from Lemma~\ref{lem:line_bundles}  that $K_{S'} -\pi'^*(K_{Y'})$ is also linearly equivalent to an effective divisor.
    If we choose an open subset $U' \subset U$ such that $\pi'(U') \subset T$ and such that $U'$ does not intersect the base locus of $K_{S'} -\pi'^*(K_{Y'})$,
    it follows from \cite[Thm.~B.3.2(e)]{HS00} that there exists a constant $C>0$ such that 

	\begin{equation} \label{eqn:split_cubic}
		N(\omega_{S'_{F'}}^{-1},U,B) \geq N(\omega_{S'_{F'}}^{-1},U',B) \geq N(\omega_{Y'_{F'}}^{-1},\pi'(U'),CB).
	\end{equation}

	Thus we have managed to reduce the counting problem to one on $Y'_{F'}$.
	As $Y'_{F'}$ is toric, by the main theorem of \cite{BT98} we see that there exists a constant $c_0>0$ such that
	$$N(\omega_{Y'_{F'}}^{-1},T,B) = c_0B(\log B)^{\rho(Y'_{F'})-1}(1+o(1)),$$
    as $B \to \infty$, where $\rho(Y'_{F'})= \rank \Pic Y'_{F'}$.
    Moreover in \cite{CLT10b}, Chambert-Loir and Tschinkel proved this asymptotic formula with respect to \emph{all} choices of adelic metric 
    on the anticanonical bundle, in particular the rational points on $T$ are equidistributed with respect to the associated
    Tamagawa measure, in the sense defined by Peyre (see \cite[Sec. 3]{Pey95}).
    It therefore follows from \cite[Prop.~3.3]{Pey95} that we also have the asymptotic formula
	$$N(\omega_{Y'_{F'}}^{-1},\pi'(U'),CB) = c_0CB(\log B)^{\rho(Y'_{F'})-1}(1+o(1)),$$
    as $B \to \infty$, since $T(\Adele_{F'})$ and $\pi'(U')(\Adele_{F'})$ have equal Tamagawa measures
	(the complement of $\pi'(U')$ in $T$ being a proper closed subvariety).
	Therefore, to finish the proof the lemma is suffices to compute $\rho(Y'_{F'})$.

    As $Y$ is a split del Pezzo surface of degree six, we have $\Pic \overline{Y} \cong \ZZ^4$ with trivial Galois action.
    Therefore by Lemma \ref{lem:Picard}, we know that $\Pic \overline{Y}'$ as a Galois module is the induced representation
    of $\ZZ^4$ with respect to the field extension $E \subset F$. In particular we have
    $\Pic \overline{Y}' \cong \ZZ^4 \oplus \ZZ^4$, with an element $\sigma \in G_{E}$ having non-trivial action
    (given by swapping the two factors of $\ZZ^4$) if and only if it has non-trivial image under the map $G_E \to  \Gal(F/E)$.
    Hence if $F \subset F'$, then $G_{F'}$ acts
    trivially and we have $\rho(Y'_{F'}) = 8$ as $\Pic Y'_{F'} \cong (\Pic \overline{Y}')^{G_{F'}}$ by
    Lemma \ref{lem:CMPic}. Otherwise $G_{F'}$ acts non-trivially and so $\rho(Y'_{F'}) = 4$. This completes the proof of the lemma.
\end{proof}
Theorem \ref{thm:counterexample} is a consequence of the following lemma.

\begin{lemma}
    Let $E$ be any number field and let $X$ be given by (\ref{eqn:BT}).
    Let $F=E(\sqrt{-3})$ and put $X'=\Res X_F$.
    Let $E \subset F'$ be a finite field extension and choose an adelic metric on $\omega^{-1}_{X'_{F'}}$. Let
    $U \subset X'_{F'}$ be an open subset and let $B>0$. Then there exists a constant $c'>0$
    such that
	$$N(\omega_{X'_{F'}}^{-1},U,B) \geq c'B(\log B)^{\rho(X'_{F'})+3},$$
    if $F \subset F'$ and
    $$N(\omega_{X'_{F'}}^{-1},U,B) \geq c'B(\log B)^{\rho(X'_{F'})+1},$$
    otherwise,  where $\rho(X'_{F'}) = \rank \Pic X'_{F'}$
\end{lemma}
\begin{proof}
	We begin by noting that as $\Pic \Xbar \cong \ZZ^2$
	with trivial Galois action, it follows as in the proof of Lemma \ref{lem:split_cubic}
	that $\rho(X'_{F'}) = 4$ if $F \subset F'$ and that $\rho(X'_{F'}) = 2$ otherwise.
	Next, consider the projection map
    \begin{align*}
        &\pi:X \to \PP^3, \qquad (x,y) \mapsto x.
    \end{align*}
    The fibres over those points with $x_0\cdots x_3 \neq 0$ are smooth diagonal cubic surfaces,
    and moreover the anticanonical bundle of these surfaces is isomorphic to the restriction of the anticanonical
    bundle on $X$. If we let
    \begin{align*}
        &\varphi:\PP^3 \to \PP^3 , \qquad (x_0:x_1:x_2:x_3) \mapsto (x_0^3:x_1^3:x_2^3:x_3^3),
    \end{align*}
    then the set $\varphi(\PP^3(F))$ is Zariski dense in $\PP^3(F)$. Moreover,
    as in the proof of \cite[Thm.~3.1]{BT96}, we see that since
    $\QQ(\sqrt{-3}) \subset F$, the fibres over those points in $\varphi(\PP^3(F))$
    with $x_0\cdots x_3 \neq 0$ are \emph{split} cubic surfaces, i.e. a Zariski dense
    set of the fibres of $\pi$  are split cubic surfaces. We want an analogous statement
    for the corresponding map $\pi'=\Res \pi: X' \to \Res \PP^3$.

    Let $\varphi'= \Res \varphi: \Res \PP^3 \to \Res \PP^3$ be the map induced by $\varphi$
    and let $p:\Res \PP^3 \to \PP^3$ be the usual universal morphism. We
    have the following commutative diagram
    \begin{equation}\label{diag:commute}
	  \begin{split}
	    \xymatrix{
               \Res \PP^3 \ar[d]^p \ar[r]^{\varphi'} & \Res \PP^3 \ar[d]^p \\
	     \PP^3 \ar[r]^{\varphi} & \PP^3 }
  	\end{split}
   \end{equation}
    Note that the fibre of $\pi'$ over a point $ x \in \Res \PP^3(E)$ is the Weil restriction of the 
    fibre of $\pi$ over the point $p(x) \in \PP^3(F)$. Also, we claim that
    $p^{-1}(\varphi(\PP^3(F)))$ is Zariski dense in $\Res \PP^3(E)$. Indeed, by the commutivity of
    (\ref{diag:commute}) we have the equality
    $p^{-1}(\varphi(\PP^3(F)))=\varphi'(\Res \PP^3(E))$.
    This later set is Zariski dense in $\Res \PP^3$ as $\varphi'$ is dominant.
    In particular, we see that there is a Zariski dense set of points in $\Res \PP^3$ whose fibres with respect to 				
    $\pi'$ are Weil restrictions of split cubic surfaces. 	
    The result therefore follows on combining Lemma \ref{lem:split_cubic} with the above calculation of $\rho(X'_{F'})$.
\end{proof}

We finish by remarking that by applying the same method to the varieties
$$
	X_{n+2}: x_0y_0^3 + x_1y_1^3 + x_2y_2^3 + x_3y_3^3=0 \subset \PP^{n+2} \times \PP^3 ,
$$
considered in \cite{BT96} for any $n>1$, one may construct counterexamples to
Manin's conjecture with arbitrary large dimension over any number field (note that the variables
$x_4,\ldots,x_{n+2}$ do not appear in this equation).

\end{document}